\documentclass[12pt]{amsart}
\usepackage{latexsym}
\usepackage{amsthm}
\usepackage{amsmath}
\usepackage{amsfonts}
\usepackage{amssymb}
\usepackage[dvips]{graphicx}
\usepackage{xypic}
\input xy
\xyoption{all}

\newtheorem{theo}{Theorem}[section]
\newtheorem{lemma}[theo]{Lemma}
\newtheorem{propo}[theo]{Proposition}
\newtheorem{assump}[theo]{Assumption}
\newtheorem{defi}[theo]{Definition}
\newtheorem{coro}[theo]{Corollary}
\newtheorem{rem}[theo]{Remark}

\newcommand\Mod{\operatorname{Mod}}

\newcommand\Set{\operatorname{\bf Set}}
\newcommand\Posl{\operatorname{{\bf Pos}_\lambda}}
\newcommand\Satl{\operatorname{{\bf Sat}_\lambda}}
\newcommand\Satx{\operatorname{{\bf Sat}_\chi}}
\newcommand\lsdk{LS^d(\ck)}
\newcommand\lsnk{LS(\ck)}

\newcommand\sgreat{\triangleright}
\newcommand\sgeq{\trianglerighteq}

\newcommand\Lin{\operatorname{\bf Lin}}

\newcommand\colim{\operatorname{colim}}

\newcommand\monst{\mathfrak {C}}
\newcommand\ck{\mathcal {K}}

\date{May 6, 2015}
 
\begin{document}
\title[Approximations of superstability in concrete accessible categories]
{Approximations of superstability in concrete accessible categories}
\author[M. Lieberman and J. Rosick\'{y}]
{M. Lieberman and J. Rosick\'{y}}
\thanks{Supported by the Grant Agency of the Czech Republic under the grant 
               P201/12/G028.} 
\address{
\newline M. Lieberman\newline
Department of Mathematics and Statistics\newline
Masaryk University, Faculty of Sciences\newline
Kotl\'{a}\v{r}sk\'{a} 2, 611 37 Brno, Czech Republic\newline
lieberman@math.muni.cz\newline
\newline J. Rosick\'{y}\newline
Department of Mathematics and Statistics\newline
Masaryk University, Faculty of Sciences\newline
Kotl\'{a}\v{r}sk\'{a} 2, 611 37 Brno, Czech Republic\newline
rosicky@math.muni.cz
}
 
\begin{abstract}
We generalize the constructions and results of Chapter 10 in \cite{Ba} to coherent accessible categories with concrete directed colimits and concrete monomorphisms.  In particular, we prove that if any category of this form is categorical in a successor, directed colimits of saturated objects are themselves saturated.
\end{abstract} 

\maketitle
 
\section{Introduction}

A longstanding preoccupation in model theory is the problem of determining, given a category of models $\ck$, the structure of the full subcategory $\Satl(\ck)$ consisting of $\lambda$-saturated models, where saturation is characterized variously in terms of the realization of syntactic types or, in more general contexts, of Galois types.  In particular, there are a number of results on the conditions under which the category $\Satl(\ck)$ is closed under unions of increasing chains---which is to say, closed under {\it directed colimits}---a pleasantly surprising property: in general, one could not reasonably hope for the union of a chain of $\lambda$-saturated models of small cofinality to be itself $\lambda$-saturated.  Results of \cite{S2} and \cite{H} guarantee that if $\ck=\Mod(T)$ with $T$ a superstable first-order theory, $\Satl(\ck)$ has precisely this property.  In the literature on abstract elementary classes, versions of this property---\textit{admitting $\lambda$-saturated unions}, in the language of \cite{Ba}---have typically been used as an analogue of superstability.  Theorem 15.8 in \cite{Ba} shows that for any AEC $\ck$ satisfying the amalgamation and joint embedding properties and containing arbitrarily large models, if $\ck$ is $\kappa$-categorical for $\kappa$ a regular cardinal then it admits $\lambda$-saturated unions for any $\lsnk<\lambda<\kappa$.  That is to say, in such an AEC, $\Satl(\ck)$ is closed under directed colimits for any $\lsnk<\lambda<\kappa$.  A recent result of \cite{BV} extends this to superstable tame abstract elementary classes with amalgamation, and shows, moreover, that if $\ck$ is $\kappa$-categorical for $\kappa$ sufficiently large, then in fact $\ck$ admits $\lambda$-saturated unions for all sufficiently large $\lambda$.  Indeed, they show that $\Satl(\ck)$ (which they denote by $\ck^{\lambda-{sat}}$) is itself an AEC.

In the context of accessible categories, $\lambda$-saturation is more usually defined as an injectivity condition: $M\in\ck$ is $\lambda$-saturated if for any morphism $f:N\to N'$ with $N$ and $N'$ $\lambda$-presentable (roughly, of size less than $\lambda$) and for any $g:N\to M$, there is a morphism $h:N'\to M$ such that $hf=g$.  The structure of $\Satl(\ck)$, here interpreted as the full subcategory of $\lambda$-saturated objects in an accessible category $\ck$, has already been studied to some degree in \cite{R}, which gives conditions under which $\Satl(\ck)$ is itself accessible.  Purely in terms of closure, it is clear that $\Satl(\ck)$ should be closed under $\lambda$-directed colimits, but there are not presently any results giving sufficient conditions for closure under $\mu$-directed colimits for $\mu<\lambda$, let alone under arbitrary directed colimits.  This paper represents a first attempt at bridging this gap.  We will not work with general accessible categories, of course: our focus will be on a slight generalization of the $\kappa$-CAECs of \cite{LR1}, which we call \textit{weak $\kappa$-CAECs}, in which we drop the assumptions of repleteness and iso-fullness included in the definition of $\kappa$-CAECs.  Such categories fall at the extreme right of the following schematic diagram of generalizations:

$$
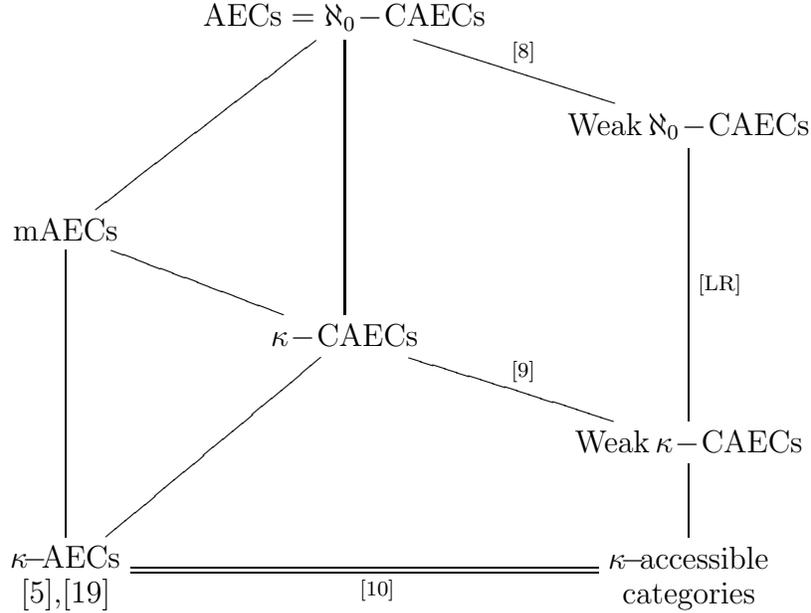
\xymatrix{ & {\rm AECs=\aleph_0\!-\!CAECs}\ar@{-}[ddl]\ar@{-}[dr]^{\rm [8]}\ar@{-}[ddd] & \\
 & & {\rm Weak\,\aleph_0\!-\!CAECs}\ar@{-}[ddd]^{\rm [LR]}\\
{\rm mAECs}\ar@{-}[ddd]\ar@{-}[dr] & & \\
 & {\rm \kappa\!-\!CAECs}\ar@{-}[dr]^{\rm [9]} & \\
 & & {\rm Weak\,\kappa\!-\!CAECs}\ar@{-}[d]\\
\txt{$\kappa\!-\!$AECs\\[5],[19]}\ar@{=}[rr]_{\rm [10]}\ar@{-}[uur] & & \txt{$\kappa\!-\!$accessible\\categories}}$$

We note that weak $\aleph_0$-CAECs are precisely the coherent accessible categories with concrete directed colimits and concrete monomorphisms considered in \cite{LR}.  Results of \cite{LR1} are stated for $\kappa$-CAECs, but are in fact valid for weak $\kappa$-CAECs.  Here we deduce the existence of universal extensions from stability in a general weak $\kappa$-CAEC, before specializing again to the case $\kappa=\aleph_0$.  Ultimately, in Theorem~\ref{bigthm}, we will prove an analogue of Theorem 10.22 in \cite{Ba} in the latter context, guaranteeing the saturation of short chains (and, by extension, arbitrary directed colimits) of saturated objects under the assumption of categoricity.  This illustrates that the property of isomorphism fullness which distinguishes AECs and weak $\aleph_0$-CAECs is not essential for this result.  Moreover, we obtain our result by a more abstract toolkit that should enable further generalizations in the future.

\section{Weak $\kappa$-CAECs}

The broad framework in which we work is a twofold generalization of the category-theoretic characterization of AECs given in \cite{LR}, namely as pairs $(\ck,U)$ where $\ck$ is a category and $U:\ck\to\Set$ is a faithful functor such that
\begin{itemize}\item $\ck$ is an accessible category with directed colimits all of whose morphisms are monomorphisms.\\
\item $(\ck,U)$ is coherent, and has concrete monomorphisms.\\
\item $U$ preserves directed colimits.\\
\item $\ck$ admits a replete, iso-full embedding into the category of structures in a canonical finitary signature $\Sigma_U$ derived from $U$.\end{itemize}

In \cite{LR1}, we note that the assumption that $U$ preserves directed colimits, i.e.~that directed colimits are concrete, may be too limiting.  Crucially, metric AECs, a generalization of AECs in which the underlying objects of the models are not sets but complete metric spaces, do not fall within this scheme: although any mAEC $\ck$ is closed under directed colimits, the forgetful functor $U:\ck\to\Set$ does not preserve directed colimits.  It does, however, preserve $\aleph_1$-directed colimits; that is, $\aleph_1$-directed colimits are concrete.  This leads naturally to the introduction of $\kappa$-concrete AECs (or $\kappa$-CAECs), which are pairs $(\ck,U)$ satisfying the axioms above with the exception that $U$ is only required to preserve $\kappa$-directed colimits.  

In this account, we also drop the assumptions of repleteness and iso-fullness.  This has already been done in some measure in \cite{LR}, where they are explicitly dispensed with, and in \cite{LR1}, where they are included in the definition of a $\kappa$-CAEC but are never used.  Having dropped this assumption, we pass from $\kappa$-CAECs to \textit{weak $\kappa$-CAECs}: 

\begin{defi}{\em We say that a pair $(\ck,U)$ consisting of a category $\ck$ and faithful functor $U:\ck\to\Set$ is a \textit{weak $\kappa$-concrete AEC}, or \textit{weak $\kappa$-CAEC}, if
\begin{enumerate}\item $\ck$ is accessible with directed colimits, and all of its morphisms are monomorphisms.\\
\item $(\ck,U)$ is coherent, and has concrete monomorphims.\\
\item $U$ preserves $\kappa$-directed colimits.\end{enumerate}}\end{defi}
When necessary, we incorporate the index of accessibility of $\ck$ into our notation: if $(\ck,U)$ is a weak $\kappa$-concrete category with $\ck$ $\lambda$-accessible, we say that $(\ck,U)$ is a \textit{weak $(\kappa,\lambda)$-concrete AEC}, or \textit{weak $(\kappa,\lambda)$-CAEC}.  We may occasionally abuse this notation by referring to $\ck$ itself as a weak $\kappa$-CAEC, or weak $(\kappa,\lambda)$-CAEC.

We collect a few basic facts about weak $\kappa$-CAECs:

\begin{rem}\label{wkcaecfacts}{\em Let $(\ck,U)$ be a large weak $(\kappa,\lambda)$-CAEC.\\
\begin{enumerate}\item Because $\ck$ is $\lambda$-accessible and has all directed colimits, it is \textit{well accessible}; that is, $\ck$ is $\mu$-accessible for all regular $\mu\geq\lambda$ (\cite{BR}~4.1).  Note that this is not true for a general $\lambda$-accessible category, e.g. $\Posl$, the category of $\lambda$-directed posets and substructure embeddings, which is $\mu$-accessible only in regular $\mu$ satisfying the sharp inequality $\mu\sgeq\lambda$.  (For more on the relation $\sgeq$, see \cite{AR} 2.11, 2.12, and 2.13.)\\
\item The presentability rank of any object $M$ in $\ck$ is a successor cardinal, say $|M|^+$.  We call $|M|$ the \textit{size of $M$}.  We note that this is a notion of size internal to the category $\ck$---the size of an object $M$ need not, in general, correspond to $|U(M)|$.  That is, $U$ need not preserve sizes. (See \cite{BR}~4.2.)\\
\item There exists a minimal cardinal $\lambda_U\geq\lambda$ such that $U$ preserves $\lambda_U$-presentable objects.  Indeed, one can show that $U$ will preserve all sizes $\mu$ such that $\mu^+\sgreat\kappa$ and $\mu^+\geq\lambda_U$. (See \cite{LR1} 4.11(1) and 4.12)\\
\item $\ck$ admits an EM-functor: there is a faithful functor $E:\Lin\to\ck$ that is faithful and preserves directed colimits.  Moreover, there is a cardinal $\lambda_E$ such that $E$ preserves sizes $\mu$ for which $\mu^+\geq\lambda_E$. (See \cite{LR1}~5.6.)
\end{enumerate}}\end{rem}

We refer readers to \cite{LR1} for further background, including a description of Galois types, saturation, and stability in this context (\cite{LR1} \S 6).  To ensure that these notions are well-behaved, we make the standard assumptions:

\begin{assump}{\em All weak $\kappa$-AECs are assumed to be large (roughly, to have arbitrarily large objects), and to satisfy the amalgamation and joint embedding properties.  Moreover, we assume that there exists a proper class of cardinals $\lambda$ with $\lambda^{<\lambda}=\lambda$, to ensure the existence of arbitrarily large monster objects, $\monst$.}\end{assump}

Note that weak $\aleph_0$-CAECs are precisely the coherent accessible categories with concrete directed colimits and concrete monomorphisms that are the principal focus of \cite{LR}.  Because $\mu^+\sgreat\aleph_0$ for any infinite cardinal $\mu$, the statement on preservations of sizes in Remark~\ref{wkcaecfacts}(2) simplifies considerably: $U$ preserves all sizes starting with $\lambda_U$.  We require one additional result for weak $\aleph_0$-CAECs, which appears as Corollary~7.9 in \cite{LR}.

\begin{theo}Let $\ck$ be a large weak $\aleph_0$-CAEC.  If $\ck$ is $\lambda$-categorical for a regular cardinal $\lambda\geq\lambda_U+\lambda_E$, then the unique object of size $\nu^+$ is saturated.\end{theo}

\section{Universal Extensions}

In our progress toward Theorem~\ref{bigthm}, we will essentially follow the argument of Chapter 10 in \cite{Ba}.  We begin by connecting stability first to the existence of 1-special extensions, defined below, then to the existence of universal extensions, which is of independent interest.

The following is Definition~10.1 in \cite{Ba}, transferred to the current context.  We work exclusively with Galois $1$-types---in the future, we will simply refer to them as ``types.''
\begin{defi}{\rm Given $M, N\in\ck$ both of size $\mu$, we say that $N$ is a \textit{$1$-special extension} of $\mu$ if there is a continuous chain of morphisms
$$\xymatrix{M_0\to M_1\to\dots\to M_i\to\dots}$$
where $M=M_0$, $M_{i+1}$ realizes all types over $M_i$, and $N=\colim_{i<\mu}M_i$.}\end{defi}

\begin{lemma}\label{exonespec}Let $(\ck,U)$ be a weak $(\kappa,\lambda)$-CAEC.  If $\ck$ is $\mu$-stable for $\mu^+\sgreat\kappa$ and $\mu\geq\lambda$, then any object $M_0$ of size $\mu$ has a $1$-special extension.\end{lemma}
\begin{proof}By $\mu$-stability, there are at most $\mu$ types over $M_0$, say $\{(f,a_i)\,|\,i<\mu\}$ with $f:M_0\to\monst$.  As $\monst$ is a $\mu^+$-directed colimit of objects of size $\mu$, and such colimits are preserved by $U$, $f$ factors through another object of size $\mu$, say as $M_0\to M_1\to\monst$, such that $a_i\in M_1$ for all $i<\mu$.  Then $M_1$ realizes all types over $M_0$.
We continue in this way, and take colimits at limit stages: note that, by 1.16 in \cite{AR}, the colimit of a chain of fewer than $\mu$ objects of size $\mu$, i.e. $\mu^+$-presentable objects, must be $\mu^+$-presentable, hence itself of size $\mu$.  Thus we build a continuous chain $\langle M_i\to M_j\,|\,i<j<\mu\rangle$, where $M_{i+1}$ realizes all $1$-types over $M_i$ for $i<\mu$.  The colimit of this chain---of size $\mu$, by the reasoning immediately above---is the desired $1$-special extension.\end{proof}

\begin{lemma}\label{sattrans}Let $(\ck,U)$ be a weak $(\kappa,\lambda)$-CAEC.  Given $g:M\to M'$, if $M'$ realizes all types over $M$, then $M'$ realizes all types over any $N$ with $f:N\to M$.\end{lemma}
\begin{proof}By joint embedding and saturation of $\monst$, we have embeddings $u:N\to\monst$, $v:M\to\monst$, and $w:M'\to\monst$.  By Remark~4.3 in \cite{LR}---which is stated there for accessible categories with concrete directed colimits, but whose proof does not in fact require any concreteness of colimits whatsoever---there is an automorphism $\bar{f}$ of $\monst$ extending $f$, in the sense that the following diagram commutes:
$$\xymatrix@=3pc{\monst\ar[r]^{\bar{f}} & \monst\\
N\ar[r]_{f}\ar[u]^u & M\ar[u]_{v}}$$
Without loss of generality, any type over $N$ is of the form $(u,a)$, with $a\in\monst$.  We use $\bar{f}$ to transform this into a type over $M$: consider $(v,U\bar{f}(a))$.  This type is realized in $M'$, meaning that there is $b\in U(M')$ and an automorphism $s$ of $\monst$ with $Us(U\bar{f}(a))=Uw(b)$, i.e. $U(s\bar{f})(a)=Uw(b)$, and such that the following diagram commutes:
$$\xymatrix@=3pc{\monst\ar[r]^s & \monst\\
M\ar[u]^v\ar[r]_g & M'\ar[u]_w}$$
Gluing the two commutative squares along $v$, we have 
$$s\bar{f}u=wgf$$
meaning that the automorphism $s\bar{f}$ witnesses the equivalence of $(u,a)$ and $(gf,b)$, meaning that $(u,a)$ is realized in $M'$.\end{proof}

We wish to show that any $1$-special extension $m:M\to\bar{M}$ (subject to certain size conditions) is universal in the following sense, again adapted only superficially from Definition 10.4 in \cite{Ba}.

\begin{defi} {\em Given a subobject $m:M\to\bar{M}$, we say that $\bar{M}$ is \textit{$\mu$-universal over $M$} if for every $h:M\to N$ with $M$ and $N$ of size at most $\mu$ there exists $t:N\to\bar{M}$ such that $th=f$.  If $M$ and $\bar{M}$ are both of size $\mu$, we simply say that $\bar{M}$ is \textit{universal over $M$}.}\end{defi}

The following generalizes a result first proven for AECs in \cite{GV}:

\begin{theo}\label{specimpliesuniv}Let $(\ck,U)$ be a weak $\kappa$-CAEC.  If $m:M\to\bar{M}$ is a $1$-special extension where $M$ and $\bar{M}$ are of size $\mu$ with $\mu^+\sgreat\kappa$ and $\mu\geq\lambda_U$, then $\bar{M}$ is universal over $M$.\end{theo}
\begin{proof}By assumption, $\bar{M}$ is the colimit of a continuous chain
$$\xymatrix{M=M_0\ar[r]^{m_{0,1}} & M_1\ar[r] & \dots\ar[r] & M_i\ar[r]^{m_{i,i+1}} & M_{i+1}\ar[r] & \dots & \bar{M}}$$
with $i<\mu$, where all Galois $1$-types over $M_i$ are realized in $M_{i+1}$ via the embedding $m_{i,i+1}:M_i\to M_{i+1}$.  Let $h:M\to N$, with $N$ also of size $\mu$.  We must construct an embedding $t$ of $N$ into $\bar{M}$ so that the following triangle commutes:
$$\xymatrix@=3pc{M\ar[rr]^m\ar[drr]_h & & \bar{M}\\
 & & N\ar[u]_t}$$
 We employ a variant of the construction used in the proof of Proposition~6.2 in \cite{LR}, essentially replacing the Galois-saturated model $K$ with the chain $M=M_0\to M_1\to\dots$  To begin, we note that $N$, being of size $\mu$ with $\mu^+\sgreat\kappa$ and $\mu\geq\lambda_U$, must satisfy $|U(N)|\leq\mu$ (see Corollary~4.12 in \cite{LR1}).  Hence we may enumerate $U(N)\setminus (Uh)(U(M))$ as $\{a_i\,|\,i<\mu\}$.  In fact, we define, for each $i<\mu$, 
 $$X_i=(Uh)(U(M))\cup\{a_k\,|\,k<i\}$$
 To facilitate our work with types, we choose, using joint embedding, maps $g_1:\bar{M}\to\monst$ and $g_2:N\to\monst$.  We build the desired map from $N$ to $\bar{M}$ element by element.  In particular, we construct a chain $\langle n_{ij}:N_i\to N_j\,|\,i<j<\mu\rangle$, and $\ck$-morphisms $f_i:N_i\to M_i$ and $v_i:N_i\to\monst$, together with a family of set maps $t_i:X_i\to U(N_i)$.  We begin with $N_0=M=M_0$, $f_0=Id_M$, $v_0=g_1h$, and $t_0=(Uh)^{-1}$.  Abusing notation slightly (the maps within and from the bottom row are set maps; the rest are in $\ck$), we aim to build the following diagram:
 $$
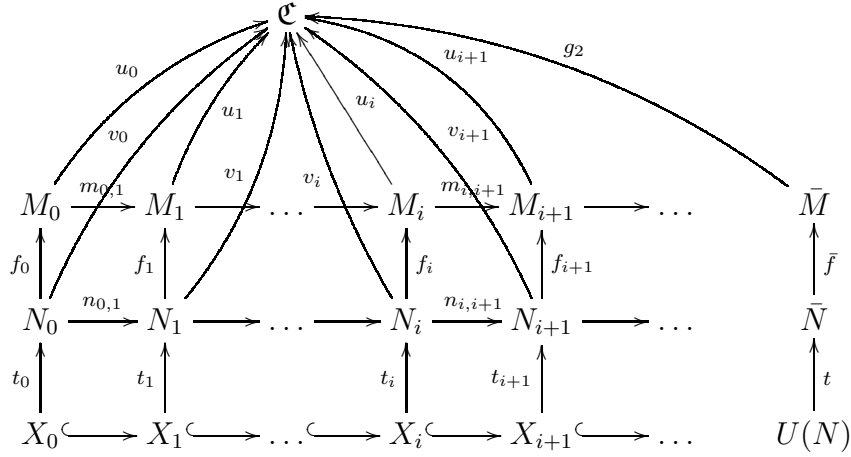
\xymatrix@=2pc{ & & \monst & & & & \\
  & & & & & & \\
M_0\ar[r]^{m_{0,1}}\ar@/^1pc/[uurr]^{u_0} & M_1\ar[r]\ar@/^/[uur]_{u_1} & \dots\ar[r] & M_i\ar[r]^{m_{i,i+1}}\ar[uul]_{u_i} & M_{i+1}\ar[r]\ar@/_1.5pc/[uull]_{u_{i+1}} & \dots & \bar{M}\ar@/_1.5pc/[uullll]_{g_2}\\
N_0\ar[r]^{n_{0,1}}\ar[u]^{f_0}\ar@/^1pc/[uuurr]^{v_0} & N_1\ar[r]\ar[u]^{f_1}\ar@/_1pc/[uuur]^{v_1} & \dots\ar[r] & N_i\ar[r]^{n_{i,i+1}}\ar[u]_{f_i}\ar@/^/[uuul]^{v_i} & N_{i+1}\ar[r]\ar[u]_{f_{i+1}}\ar@/_1pc/[uuull]_{v_{i+1}} & \dots & \bar{N}\ar[u]_{\bar{f}}\\
 X_0\ar[u]^{t_0}\ar@{^{(}->}[r] & X_1\ar[u]^{t_1}\ar@{^{(}->}[r] & \dots\ar@{^{(}->}[r] & X_i\ar[u]^{t_i}\ar@{^{(}->}[r] & X_{i+1}\ar[u]^{t_{i+1}}\ar@{^{(}->}[r] & \dots & U(N)\ar[u]_{{t}} }$$
 Here the maps $u_i:M_i\to\monst$ are just the compositions of the colimit maps $M_i\to\bar{M}$ with $g_2:\bar{M}\to\monst$, meaning that the triangles of $u$ and $m$ morphisms commute automatically.  We will ensure that all squares commute and, in addition, that $v_{i+1}n_{i,i+1}=v_i$.
 
Suppose we have constructed up to the $i$th stage, i.e. we have $N_i$, $f_i$, $v_i$, and $t_i$.  If $a_i\in U(N_i)$, we take $N_{i+1}=N_i$, and so on.  Suppose that $a_i\not\in U(N_i)$, and consider the type $(v_i,Ug_1(a_i))$ over $N_i$.  Since we have $f_i:N_i\to M_i$ and $M_{i+1}$ realizes all types over $M_i$, Lemma~\ref{sattrans} implies that this type is realized in $M_{i+1}$.  Hence there is $b\in U(M_{i+1})$ so that $(v_i,Ug_1(a_i))$ is equivalent to $(m_{i,i+1}f,b)$.  That is, there is an automorphism $s$ of $\monst$ with $su_{i+1}m_{i,i+1}f_{i+1}=v_i$ and $U(su_{i+1})(b)=U(g_1)(a_i)$.  As $u_{i+1}m_{i,i+1}=u_i$, we can simplify the first equation to $su_if_i,v_i$.
 Take $N_{i+1}$ of size $\mu$ with $c\in U(N_{i+1})$ and morphisms $n_{i,i+1}:N_i\to N_{i+1}$ and $f_{i+1}:N_{i+1}\to M_{i+1}$ so that $f_{i+1}n_{i,i+1}=m_{i,i+1}f_i$ and $Uf_{i+1}(c)=b$.  Set $v_{i+1}=su_{i+1}f_{i+1}$ (this guarantees that the upper row of squares commutes).  Notice that 
 $$v_{i+1}n_{i,i+1}=su_{i+1}f_{i+1}n_{i,i+1}=su_{i+1}m_{i,i+1}f_i=su_if_i=v_i$$
 so the triangles commute as desired.  Finally, we set $t_{i+1}=t_i\cup\{(a_i,c)\}$.  Notice that 
 $$Uv_{i+1}(c)=U(su_{i+1}f_{i+1})(c)=U(su_{i+1})(b)=Ug_1(a_i)$$
 In fact, then, $(Uv_{i+1})t_{i+1}=Ug_1$.
 
 While we might be inclined to simply take directed colimits in case $i$ is a limit ordinal, there is a slight---very slight---wrinkle: such colimits need not be concrete.  We note, however, that we may take such directed colimits freely in the chains of $N$'s and $M$'s, and must only account for the map $t_i'=\bigcup_{k<i}t_k$ which sends $X_i$ into $\bigcup_{k<i}U(N_k)$, when in fact we would prefer to map $X_i$ into $U(\colim_{k<i}N_k)$.  This is easily fixed, however: simply compose $t_i'$ with the canonical inclusion map from $\bigcup_{k<i}U(N_k)$ into $U(\colim_{k<i}N_k)$ to obtain the desired $t_i$.
 
Passing to the colimit, $(U\bar{v})\bar{t}=Ug_1$ on domain $U(N)$.  By coherence, there is a $\ck$-morphism $t:N\to\bar{M}$ with $Ut=\bar{t}$.  This $t$ is the desired embedding of $N$ into $\bar{M}$ over $M$.\end{proof}
 
This gives, immediately,
 
\begin{coro}\label{exunivext} Let $(\ck,U)$ be a weak $(\kappa,\lambda)$-CAEC.  If $\ck$ is $\mu$-stable for $\mu^+\sgreat\kappa$ and $\mu\geq\lambda+\lambda_U$, then any object $M$ of size $\mu$ has a universal extension.\end{coro}
 
Or, more particularly,
\begin{coro}\label{maecsexunivext} Let $\ck$ be an mAEC.  If $\ck$ is $\mu$-stable for $\mu^+\sgreat\aleph_1$ and $\mu\geq\lsdk+\lsdk^{\aleph_0}$, then any model of density character $\mu$ has a universal extension also of density character $\mu$.\end{coro}

We note that Theorem 2.11 in \cite{Z} establishes the latter result given stability in any $\mu$: our methods, which are fundamentally discrete rather than metric, yield a weaker result in this case.
 
 \section{Weak $\aleph_0$-CAECs: Saturation of EM-objects}\label{alephz}
 
Recall that, by Remark~2.8 in \cite{LR}, any coherent accessible category with directed colimits $\ck$ whose morphisms are monomorphisms---and therefore any weak $(\kappa,\lambda)$-CAEC---admits an EM-functor $E:\Lin\to\ck$ that is faithful, preserves directed colimits, and preserves sizes starting with some cardinal $\lambda_E$.  In this section, we focus on the conditions on $\ck$ and on a linear order $I$ that ensure the EM-object $E(I)$ is $\mu$-saturated.

The following technical result concerning the EM-functor is an approximation of Lemma~10.11 in \cite{Ba} in the context of weak $\aleph_0$-CAECs: for emphasis, we note again that these are equivalent to the coherent accessible categories with concrete directed colimits considered in \cite{LR}.  Notice that this proof removes any reference to terms and signatures: purely category-theoretic properties of the functor $E$ are sufficient.

\begin{propo}\label{emsat}Let $\ck$ be a weak $\aleph_0$-CAEC.  Suppose $\ck$ is $\lambda$-categorical for regular $\lambda>\lambda_U+\lambda_E$.  For any linear order $J$ with $\lambda_E<|J|<\lambda$ and the property that $J$ contains an increasing sequence of length $\theta^+$ for all $\lambda_E+\lambda_U<\theta<|J|$, $E(J)$ is saturated.\end{propo}
\begin{proof} Let $J$ be a linear order with the property described above, and let $M=E(J)$.  We wish to show that for any $\theta$ with $\lambda_E+\lambda_U<\theta<|J|$, $M$ is $\theta^+$-saturated.  Take such a $\theta$.  By assumption, there is $J_0\subseteq J$ of order type $\theta^+$.  Define 
$$J'=J_0+\lambda$$
as the ordered sum.  Then $|J'|=\lambda$ and since $\lambda>\lambda_E$, the object $N=E(J')$ is of size $\lambda$.  Hence it is isomorphic to the categorical object---by Corollary~7.9 in \cite{LR}, $N$ is therefore saturated.

Consider a subobject $f:M_0\to M$, with $M_0$ of size less than $\theta$.  By categoricity and joint embedding, there is an embedding $g:M_0\to N$.  Moreover, because $N$ is saturated, any type over $M_0$ is equivalent to $(g,b)$ for some $b\in U(N)$.  Notice that we may express $J'$ as the directed colimit of its finite suborders, $J'=\colim_{i\in I} J_i'$.  Since both $E$ and $U$ preserve directed colimits,
$$U(N)=UE(J')=UE(\colim_{i\in I} J_i')=\bigcup_{i\in I} UE(J_i')$$
In particular, $b\in UE(J_i')$ for some $i\in I$.  As a finite suborder of $J'$, $J_i'=J_{i,1}'+J_{i,2}'$, where $J_{i,1}'\subseteq J_0$ and $J_{i,2}'\subseteq\lambda$, and both are finite.  Take $K$ with $J_{i,1}'\subseteq K\subseteq J$ and $|K|=\theta$ so that $Uf(U(M_0))\subseteq UE(K)$.  As $J_0$ is of order type $\theta^+$, there is room to choose $J''\subseteq J_0$ with $K+J_{i,2}'$ order-isomorphic to $K+J''$ over $K$.  This induces an isomorphism $E(K+J_{i,2})\to E(K+J'')$ over $E(K)$, hence over $Uf(UM_0)$.  The image of $b$ under this isomorphism---call it $a$---lies in $M$ and satisfies $(f,a)\sim(g,b)$.\end{proof}

By a slight generalization of this argument, we also have:

\begin{propo}\label{emmusat} Let $\ck$ be a weak $\aleph_0$-CAEC.  Suppose $\ck$ is $\lambda$-categorical for regular $\lambda>\lambda_U+\lambda_E$.  For any linear order $J$ with $\lambda_E<|J|<\lambda$ and the property that $J$ contains an increasing sequence of length $\theta^+$ for all $\lambda_E+\lambda_U<\theta<\mu$ for some $\mu\leq|J|$, $E(J)$ is $\mu$-saturated.\end{propo}

We note that these propositions tell us a great deal about the saturation of EM-objects based entirely on the properties of the associated linear orders.  The following is, essentially, Corollary~10.14 in \cite{Ba}:

\begin{coro}\label{emsatex}Let $\ck$ be a weak $\aleph_0$-AEC, and suppose that $\ck$ is $\lambda$-categorical for some regular $\lambda>\lambda_U+\lambda_E$.
\begin{enumerate}\item For any $\mu$ with $\lambda_E+\lambda_U<\mu\leq\lambda$, $E(\mu)$ and $E(\mu^{<\omega})$ are saturated.\\
\item If $E(I_0)$ is $\mu$-saturated with $I_0$ a linear order satisfying the conditions of Proposition~\ref{emmusat}, then for any extension $I$ of $I_0$, $E(I)$ is $\mu$-saturated as well.\end{enumerate}\end{coro}

\section{Weak $\aleph_0$-CAECs: Limit Objects and Saturation}
 
We now define the notion of a limit model in our context.

\begin{defi}{\em Let $\langle M_i\to M_j\,|\,i\leq j<\delta\rangle$ be a continuous chain of objects of size $\mu$ with $\delta\leq\mu^+$.  We say that it is a \textit{$(\mu,\delta)$-chain} if $M_{i+1}$ $\mu$-universal over $M_i$.  In case $\delta$ is a limit ordinal, we say the colimit $\bar{M}$ is a \textit{$(\mu,\delta)$-limit object}.}\end{defi}

An analogue of saturated models, limit models were introduced in \cite{S}, and have subsequently been used in a number of contexts within abstract model theory: in the analysis of categoricity in AECs with no maximal models (\cite{SV} and \cite{Vn}), for example.  In particular, the uniqueness of limit models has been used heavily in the development of the stability theory of AECs and mAECs (e.g. in \cite{BG} and \cite{ViZ}, respectively).  This uniqueness appears as an axiom for good frames in \cite{S}, and has been proven from superstability-like assumptions in AECs and mAECs (e.g. in \cite{BG} and \cite{ViZ1}, respectively).  Here, we require only a very mild version of uniqueness:

\begin{lemma}\label{uniqlim} Let $(\ck,U)$ be a weak $\aleph_0$-CAEC.  Let $\delta$ be a limit ordinal with $\delta<\mu^+$.  Let $\bar{M}$ and $\bar{N}$ be $(\mu,\mu\times\delta)$-limits over $M_0$ and $N_0$, respectively, with $M_0\cong N_0$.  Then $\bar{M}\cong\bar{N}$ and, moreover, the chains are term-by-term isomorphic on all terms indexed by $\alpha\times i$ with $\alpha\leq\mu$ a limit ordinal and $i<\delta$.\end{lemma}
\begin{proof} The proof proceeds by a back-and-forth construction.  As the current context is slightly exotic, we sketch the argument.  In particular, we will consider the chain up to $M_{\omega,0}$, the first interesting case.  For notational simplicity, we omit the second subscript in writing the chain maps $\phi_{i,j}^M:M_{i,0}\to M_{j,0}$ and $\phi_{i}^M:M_{i,0}\to\bar{M}$, as well as the chain maps $\phi_{i,j}^N$ and $\phi_j^N$, and the system of back and forth maps.  Let $f_{0}:M_{0,0}\to N_{0,0}$ be an isomorphism.  By universality of $M_{1,0}$ over $M_{0,0}$, there is a morphism $g_1:N_{0,0}\to M_{1,0}$ such that $g_1f_0=\phi_{0,1}^M$.  By universality of $N_{1,0}$ over $N_{0,0}$, there is a morphism $f_1:M_{1,0}\to N_{1,0}$ such that $f_1g_1=\phi_{0,1}^N$.  The process continues, yielding the following diagram in which all triangles commute:
$$\xymatrix{M_{0,0}\ar[d]_{f_0}\ar[r]^{\phi_{0,1}^M} & M_{1,0}\ar[r]^{\phi_{1,2}^M}\ar[d]^{f_1} & M_{2,0}\ar[d]^{f_2} \dots & M_{i,0}\ar[r]^{\phi_{i,i+1}^M}\ar[d]^{f_i} & M_{i+1,0}\ar[d]^{f_{i+1}}\dots & M_{\omega,0}\ar@/_/[d]_{f_\omega} \\
N_{0,0} \ar[r]_{\phi_{0,1}^N}\ar[ur]_{g_1} & N_{1,0}\ar[r]_{\phi_{1,2}^N}\ar[ur]_{g_2} & N_{2,0} \dots & N_{i,0}\ar[r]_{\phi_{i,i+1}^N}\ar[ur]_{g_{i+1}} & N_{i+1,0}\dots & N_{\omega,0}\ar@/_/[u]_{g_\omega}}$$
Here $f_{\omega}:M_\omega\to N_\omega$ is the map induced by the morphisms $\phi_{i,\omega}^N f_i:M_i\to N_\omega$, and $g_\omega:N_\omega\to M_\omega$ is the map induced by the morphisms $\phi_{i+1,\omega}g_{i+1}:N_i\to M_\omega$.  Consequently, we have
$$\phi_{i,\omega}^Nf_{i}=f_{\omega}\phi_{i,\omega}^M\hspace{10 mm}\hspace{10 mm}\phi_{i+1,\omega}^Mg_{i+1}=g_{\omega}\phi_{i,\omega}^N$$
as well as 
$$g_{i+1}f_i=\phi_{i,i+1}^M\hspace{20 mm}f_{i+1}g_{i+1}=\phi_{i,i+1}^N$$
by construction.  We claim that $f_\omega$ and $g_\omega$ are inverses.  To begin, we compute:
$$\begin{array}{rcl}\phi_{i+1,\omega}^Mg_{i+1} & = & g_{\omega}\phi_{i,\omega}^N\vspace{2 mm}\\
\phi_{i+1,\omega}^Mg_{i+1}f_i & = & g_{\omega}\phi_{i,\omega}^Nf_i\vspace{2 mm}\\
\phi_{i+1,\omega}^M\phi_{i,i+1}^M & = & g_{\omega}f_\omega\phi_{i,\omega}^M\vspace{2 mm}\\
\phi_{i,\omega}^M & = & g_{\omega}f_\omega\phi_{i,\omega}^M\end{array}$$
Since the colimit maps $\phi_{i,\omega}^M$ form a jointly epimorphic family, it follows that 
$$g_{\omega}f_{\omega}=\mbox{Id}_{M_\omega}$$
Similarly, $f_\omega g_\omega\phi_{i-1,\omega}^N=\phi_{i-1,\omega}^N$.  Since the $\phi_{i-1,\omega}^N$ are jointly epimorphic as well, we have
$$f_{\omega}g_{\omega}=\mbox{Id}_{N_\omega}$$
Beyond stage $\omega\times 0$, we resume the back-and-forth sequence, and continue this process.\end{proof}

In case $\ck$ is categorical this gives a characterization of limits over EM-objects, which slightly generalizes Lemma~10.16(3) in \cite{Ba}.

\begin{lemma}\label{emlims}Let $(\ck,U)$ be a weak $\aleph_0$-CAEC that is $\lambda$-categorical for regular $\lambda>\lambda_U+\lambda_E$, and let $\lambda_E<\mu<\lambda$.  Let $\delta<\mu^+$ be a limit ordinal and let $I=\mu^{<\omega}$.  Every $(\mu,\delta)$-chain over $E(I)$ is isomorphic to $E(I\times\delta)$.\end{lemma}
\begin{proof}We first prove $\langle E(I\times\alpha)\,|\,\alpha<\delta\rangle$ is a $(\mu,\delta)$-chain over $E(I)$.  In particular, we show that $E(I\times(\alpha+1))$ is $\mu$-universal over $E(I\times\alpha)$.  Consider $f:E(I\times\alpha)\to M$ with $M$ of size $\mu$.  By saturation of $E(I\times\lambda)$---which follows from categoricity in $\lambda$, again by Corollary~7.9 in \cite{LR}---there is a morphism $g:M\to E(I\times\lambda)$ such that the following diagram commutes:
$$\xymatrix{E(I\times\alpha)\ar[r]^{f}\ar[dr]_{\bar{i}} & M\ar[d]_{g}\\
 & E(I\times\lambda)}$$
 where $\bar{i}$ is the $\ck$-morphism induced by the inclusion $i:I\times\alpha\to I\times\lambda$.  As $I\times\lambda$ can be expressed as a $\mu^+$-directed colimit of suborders $I\times\alpha\subseteq Y\subseteq I\times\lambda$ of size $\mu$, and the functor $E$ preserves both notions, $E(I\times\lambda)$ is a $\mu^+$-directed colimit of the $E(Y)$, all of which are of size $\mu$.  As $M$ is $\mu^+$-presentable, $g$ factors through some $E(Y)\to E(I\times\lambda)$.  So we have
 $$\xymatrix{E(I\times\alpha)\ar[rr]^{f}\ar[dr]_{\bar{j}} & & M\ar[dd]_{g}\ar[dl]^{g'}\\
 & E(Y)\ar[dr] & \\
 & & E(I\times\lambda)}$$
 where the upper triangle commutes.  So $M$ embeds in $E(Y)$ over $E(I\times\alpha)$, hence it suffices to show that $E(Y)$ embeds in $E(I\times(\alpha+1))$.  This is simple: $Y$ consists of a disjoint union of $Y_0\subseteq I\times\alpha$ and $Y_1\subseteq I\times\lambda\setminus I\times\alpha$ of size at most $\mu$.  By universality of $I=\mu^{<\omega}$, $Y_1$ embeds in the $(\alpha+1)$st copy of $I$.  So we have an induced morphism $E(Y)\to E(I\times(\alpha+1))$ that fixes $E(I\times\alpha)$.  Hence $\langle E(I\times\alpha)\,|\,\alpha<\delta\rangle$ is a $(\mu,\delta)$-chain over $E(I)$, as claimed.
 
 The conclusion now follows from the uniqueness of $(\mu,\mu\times\delta)$-limits over $E(I)$ proven in Lemma~\ref{uniqlim}.\end{proof}
 
 We require one more flavor of uniqueness for limit objects, which, like Lemma~\ref{uniqlim} (or Lemma~10.8 in \cite{Ba}), can be obtained via a simple back-and-forth argument.
 
 \begin{lemma} If $M$ is of size $\mu$ and $\bar{M}$ and $\bar{N}$ are, respectively, $(\mu,\delta)$- and $(\mu,\mu\times\delta)$-limits over $M$ with $\delta<\mu^+$, then $\bar{M}\cong\bar{N}$.\end{lemma}
 
 \begin{theo}\label{bigthm} Let $(\ck,U)$ be a weak $\aleph_0$-CAEC that is categorical in $\lambda^+$, and let $\lambda_U+\lambda_E<\chi<\lambda^+$ with $\chi$ a limit cardinal.  Then the colimit of any continuous $\delta$-chain of $\chi$-saturated objects with $\delta<\chi^+$ is $\chi$-saturated.\end{theo}
 \begin{proof}By Corollary~\ref{emsatex}, $\chi$-saturated objects exist.  Let $N$ be the colimit of a chain $\langle \phi_{i,j}:N_i\to N_j\,|\,i\leq j<\delta\rangle$, where each $N_i$ is $\chi$-saturated.  Consider $f:M\to N$ with $M$ of size $\kappa<\chi$.  We wish to show that any type over $M$ is realized in $N$.
 
By concreteness of directed colimits, 
$$U(N)=\bigcup_{i<\delta} U\phi_i(U(N_i))$$
and $Uf(U(M))$ is a subset.  Let
$$I=\{i<\delta\,|\,U\phi_{i+1}(U(N_{i+1}))\cap Uf((U(M))\setminus U\phi_i(U(N_i))\neq\emptyset\}$$
Note that we can list this index set as $I=\{i_\alpha\,|\,\alpha<\delta'\}$ where 
$$\delta'<\max(\kappa^+,{\rm cf}(\delta)^+)<\chi.$$
We can express $U(M)$ as the union of an increasing $\delta'$-chain of sets 
$$X_\alpha=\{x\in U(M)\,|\,Uf(x)\in U\phi_i(U(N_i))\}$$
For each $\alpha<\delta'$, let $f_\alpha:X_\alpha\to N_{i_\alpha}$ such that $(U\phi_{i_\alpha})f_\alpha=Uf$.  We build a continuous chain $\langle \psi_{i,j}:M_\alpha\to M_\beta\,|\,\alpha\leq\beta<\delta'\rangle$ such that each $M_i$ is $\kappa^+$-saturated, each $\psi_{\alpha,\alpha+1}:M_\alpha\to M_{\alpha+1}$ is $\mu$-universal.  Moreover, we wish to build a series of set maps $t_\alpha:X_\alpha\to U(M_\alpha)$ and $\ck$-morphisms $g_\alpha:M_{\alpha}\to N_{i_\alpha}$ so that all cells of the following diagram commute:
$$\xymatrix{
N_{i_0}\ar[r]^{\phi_{i_0,i_1}} & N_{i_1}\ar[r]^{\phi_{i_1,i_2}} & N_{i_2}\ar[r] & \dots & N\\
M_0\ar[r]_{\psi_{0,1}}\ar[u]^{g_0=Id} & M_1\ar[r]_{\psi_{1,2}}\ar[u]^{g_1} & M_2\ar[r]\ar[u]^{g_2} & \dots & \bar{M}\ar[u]^{g}\\
X_0\ar[r]\ar[u]^{t_0}\ar@/_1pc/[uu]_>{f_0} & X_1\ar[r]\ar[u]^{t_1}\ar@/_1pc/[uu]_>{f_1} & X_2\ar[r]\ar[u]^{t_2}\ar@/_1pc/[uu]_>{f_2} & \dots & U(M)\ar[u]^{t}\ar@/_1pc/[uu]_>{f}
}$$
We begin with $M_0=N_{i_0}$ and $g_0$ the identity, with $t_0=f_0$.  By $\kappa^{++}$-saturation of $N_1$, we can find $\psi_{0,1}:M_0\to M_1$ and $g_1:M_1\to N_{i_1}$ where $M_1$ is of size $\kappa^+$, is $\kappa^+$-saturated, and $\psi_{0,1}:M_0\to M_1$ is $\mu$-universal.  Moreover, we may do so in such a way as to ensure that $f_1(X_1)\subseteq Ug_1(M_1)$.  Choose $t_1$ so that $(Ug_1)t_1=f_1$.  Notice that $f_0$ and $f_1$ are compatible with $\phi_{0,1}$ by design, meaning that the outer rectangle of the diagram below commutes, as does the upper square.
$$\xymatrix{
N_{i_0}\ar[r]^{\phi_{i_0,i_1}} & N_{i_1}\\
M_0\ar[r]_{\psi_{0,1}}\ar[u]^{g_0=Id} & M_1\ar[u]^{g_1}\\
X_0\ar[r]\ar[u]^{t_0} & X_1\ar[u]^{t_1}
}$$
as all morphisms are mono, this implies that the lower square commutes as well.  We proceed in the same way at each successor stage.  For limit $\alpha$, take colimits.  
The $t_\alpha$ induce a set map $t:U(M)\to U(\bar{M}))$, the $g_\alpha$ induce a $\ck$-morphism $g:\bar{M}\to N$, and the construction ensures that $(Ug)t=Uf$.  By coherence, there is a $\ck$-morphism $\bar{t}:M\to\bar{M}$ so that $U\bar{t}=t$.

If $\kappa<{\rm cf}(\delta)$, $M$ embeds in some $M_{\alpha+1}$ with $\alpha<\delta'$.  By $\kappa^+$-saturation of $M_{\alpha+1}$, any type over $M$ is realized in $M_{\alpha+1}$, hence in $\bar{M}$, hence in $N$.  Otherwise, $\langle M_\alpha\,|\,\alpha<\delta'\rangle$ is a $(\kappa^+,\delta')$-chain over $M_0$. 
Without loss of generality (but possibly at the cost of deleting $M_1$) we may replace $M_0$ by $E((\kappa^+)^{<\omega})$.  Being a $(\kappa^+,\delta')$-limit over $E((\kappa^+)^{<\omega})$, 
$$\bar{M}\cong E((\kappa^+)^{<\omega}\times\delta')$$
 by Lemma~\ref{emlims}.  Hence, by Corollary~\ref{emsatex}(2), $\bar{M}$ is $\kappa^+$-saturated, and we are finished.\end{proof}
 
 As $\Satx(\ck)$, the full subcategory of $\ck$ consisting of $\chi$-saturated models, is certainly closed under chains of length (or, rather, cofinality) at least $\chi^+$, Theorem~\ref{bigthm} implies that $\Satx(\ck)$ is closed under colimits of arbitrary chains, hence, in fact, under arbitrary directed colimits:
 
 \begin{coro}Let $(\ck,U)$ be a weak $\aleph_0$-CAEC that is categorical in $\lambda^+$, and let $\lambda_U+\lambda_E<\chi<\lambda^+$ with $\chi$ a limit cardinal.  Then $\Satx(\ck)$ is closed under directed colimits.\end{coro}

\end{document}